\documentclass{amsart}

\usepackage{amssymb,amsthm,amsmath}
\usepackage{enumerate}
\usepackage{cite}
\usepackage[hidelinks]{hyperref}
\usepackage{dsfont}
\usepackage{tikz}
\usetikzlibrary{arrows}

\newtheorem{lemma}{Lemma}
\newtheorem{theorem}[lemma]{Theorem}
\newtheorem{corollary}[lemma]{Corollary}
\newtheorem{prop}[lemma]{Proposition}

\theoremstyle{definition}
\newtheorem{definition}[lemma]{Definition}
\newtheorem{fact}[lemma]{Fact}

\newtheorem{numberedRemark}[lemma]{Remark}
\newtheorem{example}[lemma]{Example}

\theoremstyle{definition}

\theoremstyle{definition}
\newtheorem*{notation}{Notation}

\def \Nats {\mathds{N}}
\newcommand{\set}[1]{\{#1\}}
\newcommand{\setarg}[2]{\{#1\ |\ #2\}}
\newcommand{\setcol}[2]{\{#1 : #2\}}

\def \col {\mathfrak{c}}

\def \domcolor {\widetilde{\col}}
\def \descolor {\hat{\col}}
\newcommand{\pairs}[1]{[#1]^2}

\DeclareMathOperator{\CB}{CB}
\newcommand{\CBr}[1]{\CB(#1)}

\DeclareMathOperator{\tree}{T}
\newcommand{\subtree}[2]{\tree^{#1}(#2)}

\newcommand{\treeorder}{<^*}

\newcommand{\treedescendant}{\lhd^*}

\DeclareMathOperator{\Fan}{Fan}
\newcommand{\subfan}[1]{\Fan^{-}(#1)}

\newcommand{\ramseynotation}[4]{#1 \to_{cl} (#2)^{#3}_{#4}}
\newcommand{\closedramseynumber}[2]{R^{cl}(#1,#2)}

\DeclareMathOperator{\ordertype}{\mathit{ord}}
\newcommand{\otp}[1]{\ordertype(#1)}

\newcommand{\neighbours}[1]{N(#1)}

\newcommand{\neighboursInSet}[2]{N(#1)\cap #2}

\newcommand{\layer}[2]{L^{#1}_{#2}}

\newcommand{\ordPersFunc}[1]{\rho_{#1}}

\newcommand{\oppress}{\mathrel{\perp}}
\newcommand{\harass}{\mathrel{_\omega{\perp}}}

\newcommand{\subcof}{\mathrel{\subseteq_{\rm cof}}}

\begin{document}
\title{The closed ordinal Ramsey number $R^{cl}(\omega^2,3) = \omega^6$}
\author{Omer Mermelstein}
\address[1]{Department of Mathematics\\
Ben-Gurion University of the Negev\\
Beer-Sheva 8410501, Israel}
\address[2]{Department of Mathematics\\
University of Wisconsin--Madison\\
WI 53706, USA}
\email{omer@math.wisc.edu}
\date{}

\begin{abstract}
Closed ordinal Ramsey numbers are a topological variant of the classical (ordinal) Ramsey numbers. We compute the exact value of the closed ordinal Ramsey number $R^{cl}(\omega^2,3) = \omega^6$.
\end{abstract}

\maketitle

\section{introduction}
For ordinals $\beta$ and $\alpha$ write $\ramseynotation{\beta}{\alpha_0,\alpha_1}{2}{}$
to mean that for every pair-colouring $\col: [\beta]^2\to \set{0,1}$ there exist some $i\in\set{0,1}$ and $X\subseteq \beta$ of order type $\otp{X} = \alpha_i$ such that $X$ is closed in its supremum, and $[X]^2\subseteq \col^{-1}(\set{i})$. Should such an ordinal exist, let $R^{cl}(\alpha_0,\alpha_1)$ denote the least such ordinal. Call $R^{cl}(\alpha_0,\alpha_1)$ the \emph{closed ordinal Ramsey number of $(\alpha_0,\alpha_1)^2$}.

Caicedo and Hilton \cite[Section 7]{CaicedoHilton} proved the upper bound $\closedramseynumber{\omega^2}{k} \leq \omega^\omega$, for every natural $k>0$. For $k=3$, the existing lower bound $\closedramseynumber{\omega^2}{3} \geq \omega^{3}$ is a consequence of \cite[Proposition 3.1]{CaicedoHilton}. In this paper, we will calculate the exact value $\closedramseynumber{\omega^2}{3} = \omega^6$.

We achieve the bound $\closedramseynumber{\omega^2}{3}\leq\omega^6$ by a combinatorial analysis of any arbitrary ``canonical'' pair-colouring of $\omega^6$ in two colours. Canonical colourings were presented and discussed in \cite{MermOrd}, where it was shown that, for our purposes, every pair-colouring can be assumed to be canonical.

The bound $\closedramseynumber{\omega^2}{3}\geq\omega^6$ is achieved by proving the more general result: for every natural $k$, $\closedramseynumber{\omega^{k+1}}{3}\geq \omega^{5k+1}$. This result is given by a single colouring $\col:\pairs{\omega^\omega}\to \set{0,1}$ such that for each $k\in\omega$ and $\theta<\omega^{5k+1}$, the restriction $\col\upharpoonright\pairs{\theta}$ demonstrates $\closedramseynumber{\omega^{k+1}}{3} > \theta$.

For a history of partition relations and Rado's arrow notation see \cite{HajLar}. The ordinal partition calculus was introduced by Erd\H{o}s and Rado in \cite{ErdRad}, and topological partition calculus was considered by Baumgartner in \cite{Baum}. Baumgartner's work was continued in recent papers on topological (closed) ordinal partition relations by Hilton, Caicedo-Hilton, and Pi\~{n}a, see \cite{HilPige},\cite{CaicedoHilton}, and the sequence of works starting with \cite{Pina}. See also \cite{OW19} and the author's \cite{MermOrd}.

\section{preliminaries}
We use lowercase greek letters $\alpha,\beta,\gamma,\dots$ to denote ordinals. For $B$ a set of ordinals, write $\ordertype(B)$ for the order-type of $B$. Write $A\subcof B$, to mean that $A$ is a cofinal subset of $B$.

For any nonzero $\alpha$ there exist a unique $l\in\Nats$, a sequence of ordinals $\gamma_1 >\dots > \gamma_l$, and a sequence of nonzero natural numbers $m_1,\dots,m_l$ such that
\[
\alpha = \omega^{\gamma_1}\cdot m_1 + \omega^{\gamma_2}\cdot m_2 + \dots + \omega^{\gamma_l}\cdot m_l.
\]
Call this representation of $\alpha$ the \emph{Cantor normal form} of $\alpha$. The \emph{Cantor-Bendixson rank} (CB rank) of $\alpha$ is $\gamma_l$, and is denoted $\CBr{\alpha}$. For the ordinal $\alpha=0$, we define $\CBr{0}=0$.

We say that $\beta\treeorder\alpha$ whenever $\alpha = \beta + \omega^\gamma$ for some nonzero ordinal $\gamma$ with $\gamma > \CBr{\beta}$. Equivalently, for some $\gamma > \CBr{\beta}$, $\alpha$ is the least ordinal of $\CB$ rank $\gamma$ with $\beta\leq \alpha$. We write $\beta\treedescendant\alpha$ if $\alpha$ is the unique immediate successor of $\beta$ in $\treeorder$. Denote $\subtree{}{\alpha} = \set{\alpha}\cup\setarg{\beta}{\beta\treeorder\alpha}$ and $\subtree{=n}{\alpha}=\setarg{\beta\in\subtree{}{\alpha}}{\CBr{\beta}=n}$. If $\CBr{\alpha}$ is a successor ordinal, denote $\subfan{\alpha} = \setarg{\beta}{\beta\treedescendant\alpha}$.

It is useful to visualize $\omega^k+1$ under the $\treedescendant$ relation as an $\omega$-regular (bar the terminal nodes), rooted, directed tree of height $k+1$. The root is $\omega^k$, the unique point of CB rank $k$. The descendants of the root are $\setcol{\omega^{k-1}\cdot i}{i\in\omega}$, all the points of CB rank $k-1$, and so on. The leaves are the points of CB rank $0$. In line with the standard order on ordinals, it is preferable to visualize the root as being on top and the leaves on the bottom. Then, the $n$-th level corresponds to the points of CB rank $n$. See Figure \ref{figure} for a visualization.

\begin{figure}[h]
\caption{A schematic of $(\omega^3 +1,\treedescendant)$}
\label{figure}
\begin{tikzpicture}[x=8,y=30]
\begin{scope}[>=latex]
\foreach \i in{0}
{
	\node at (\i,4) {$\omega^3$};
	
	\node at (\i-14, 2.5) {$\omega^2$};
	\node at (\i, 2.5) {$\omega^2\cdot 2$};
	\node at (\i+14, 2.5) {$\omega^2\cdot 3$};
	
	\node at (\i-14 -4,1) {$\omega$};
	\node at (\i-14,1) {$\omega\cdot 2$};
	\node at (\i-14 +4,1) {$\omega\cdot 3$};

	\foreach \j in{0,14}
	{
		\draw [<-, shorten <=7pt, shorten >=7pt] (\i,4) -- (\i+\j,2.5);
		\foreach \k in{-4,0,4}
			{
			\node[fill=black,circle, inner sep=1pt] at (\i+\j +\k,1) {};
			\draw [<-, shorten <=7pt] (\i+\j,2.5) -- (\i+\j +\k,1);
			\foreach \l in{-1,0,1}
				{
				\node[fill=black,circle, inner sep=0.8pt] at (\i+\j+\k+\l,0) {};
				\draw [<-, shorten <=3.5pt] (\i+\j+\k,1) -- (\i+\j +\k+\l,0);
				}
			\node[label=right:\tiny $...$] at (\i+\j+\k+1/2,0) {};
			}
		\node[label=right:$\dots$] at (\i+\j+3.5,1) {};
	}
	\foreach \j in{-14}
	{
		\draw [<-, shorten <=7pt, shorten >=7pt] (\i,4) -- (\i+\j,2.5);
		\foreach \k in{-4,0,4}
			{
			\draw [<-, shorten <=7pt,shorten >=4pt] (\i+\j,2.5) -- (\i+\j +\k,1);
			\foreach \l in{-1,0,1}
				{
				\node[fill=black,circle, inner sep=0.8pt] at (\i+\j+\k+\l,0) {};
				\draw [<-, shorten <=3.5pt] (\i+\j+\k,1) -- (\i+\j +\k+\l,0);
				}
			\node[label=right:\tiny $...$] at (\i+\j+\k+1/2,0) {};
			}
		\node[label=right:$\dots$] at (\i+\j+4.5,1) {};
	}
	\node[label=right:\Huge$\dots$] at (\i+15,2.5) {};
}
\end{scope}
\end{tikzpicture}
\end{figure}

It is advisable that the reader takes a moment to locate in the figure the objects $\subtree{}{\alpha}$, $\subtree{=n}{\alpha}$, $\subfan{\alpha}$ for some $\alpha\leq\omega^3$, $n\leq \CB(\alpha)$. Another suggestion is to find in the figure a few copies of $\omega^2$ -- some closed in their supremum and some not.

When we ``thin out'' a set of ordinals $X$ of order type $\omega^k$, we mean that we take $Y\subseteq X$ such that $(Y,\treeorder)$ is isomorphic to $(X,\treeorder)$. Preserving the relation $\treeorder$ guarantees that if $X$ was closed in its supremum, then also $Y$ is, and furthermore $\otp{Y} = \otp{X}$. Unlike the actual order on the points, the order type $\omega^k$ can be read off of $\treeorder$.

\medskip
Let $G=(V,E)$ be a graph. We only consider graphs where the edge relation is symmetric. We identify the graph $G$ with a colouring $\col:[V]^2\to \set{0,1}$ by taking $\col(v,u) = 1$ if and only if $(v,u)\in E$. For $v\in V$, denote $\neighbours{v} = \setarg{u\in V}{(v,u)\in E}$. For $U\subseteq V$, denote $\neighbours{U} = \bigcup_{v\in U}\neighbours{v}$. We say that a set of vertices $U\subseteq V$ is a \emph{clique} if $\pairs{U}\subseteq E$ and that it is \emph{independent} if $\pairs{U}\cap E = \emptyset$.

\begin{definition}
	Let $G = (\delta,E)$ be some graph on an ordinal $\delta$. Let $A,B\subseteq \delta$ be infinite, disjoint and without maxima.
	\begin{itemize}
		\item
		Write $A\oppress B$ to mean that for all $X$, if $X\subcof A$, then $B\setminus\neighbours{X}$ is finite.
		\item
		Write $A\harass B$ if $A\oppress B$ and in addition $\neighboursInSet{a}{B}$ is finite for every $a\in A$.
	\end{itemize}
\end{definition}

\begin{lemma}[\hspace{-0.001pt}\cite{MermOrd}, Lemma 4.2]
	\label{harassmentLemma}
	Let $G =(\delta,E)$ for $\delta$ countable. Let $A,B\subseteq \delta$ be such that $A\harass B$. Then there is some $A_0\subcof A$ and some $B_0\subseteq B$, cofinite in $B$, such that $\neighboursInSet{b}{A_0}$ is cofinite in $A_0$, for all $b\in B_0$.
\end{lemma}

For the full definition of a canonical colouring, $\domcolor$ and $\descolor$, refer to subsection 2.2 and section 3 of \cite{MermOrd}. In particular, see definitions 3.3, 3.9 and 3.10 there. For this paper, we specialize the definition and results we use to the special case of a colouring of $\omega^k$ in two colours, for a natural $k$.

Fix $\mathfrak{F}$ to be the filter of cofinite subsets of $\omega$. Define $\mathfrak{F}^1 = \mathfrak{F}$ and $\mathfrak{F}^{n+1}$ on $\omega^{n+1}$ inductively by taking the product filter on $\omega^{n+1}\cong\omega\times\omega^n$. That is, $X\in \mathfrak{F}^{n+1}$ if and only if $\setcol{\alpha\in\omega}{\setcol{\beta\in\omega^n}{(\alpha,\beta)\in X}\in \mathfrak{F}^n}\in \mathfrak{F}^1$. For $X$, a set of ordinals with $\otp{X} = \omega^k$ witnessed by $\rho:X\to \omega^k$, we say that $Y\subseteq X$ is a $k$-\emph{large} set in $X$ if $\rho[Y]\in \mathfrak{F}^k$. If $k$ is clear from context, we may omit it and simply say that $Y$ is \emph{large} in $X$.

Say that a graph (colouring) on $\omega^k$ is \emph{canonical} if
\begin{enumerate}[(i)]
\item 
For every $\theta \leq \omega^k$, $\alpha<\omega^k$ and $l\leq k$, either $\subtree{=l}{\theta}\cap \neighbours{\alpha}$ or $\subtree{=l}{\theta}\setminus \neighbours{\alpha}$ is an $l$-large set in $\subtree{=l}{\theta}$.
\item
Additionally, if $\alpha=\theta$, then the $l$-large set in $\subtree{=l}{\theta}$ above is the entirety of $\subtree{=l}{\theta}$.
\item
Finally, for $\theta = \omega^k$, whether the $l$-large set in $\subtree{=l}{\theta}$ is contained in $\neighbours{\alpha}$ or disjoint from $\neighbours{\alpha}$ is determined only by $\CB(\alpha)$.
\end{enumerate}
For $k>j>l$, denote by $\descolor(j,l)\in\set{0,1}$ the ``colour'' by which every $\alpha\in \subtree{=j}{\omega^k}$ is connected to every $\beta\in \subtree{=l}{\alpha}$.
\\
For $k>j,l$, denote by $\domcolor(j,l)\in\set{0,1}$ the ``colour'' by which every $\alpha\in \subtree{=j}{\omega^k}$ is connected to an $l$-large set in $\subtree{=l}{\omega^k}$.

\begin{example}
The following is the edge-set of a canonical graph on $\omega^2$:
\[
\setcol{\set{\omega\cdot k,\omega\cdot l + n}}{l>k>n>0} \cup
\setcol{\set{\omega\cdot k + k',\omega\cdot l + l'}}{k<l, l'>k'>0} 
\]
For this graph: $\descolor(1,0) = 0$, $\domcolor(1,1) = 0$, $\domcolor(0,0) = 1$, $\domcolor(1,0) = 0$, $\domcolor(0,1) = 0$.
\end{example}

\begin{numberedRemark}
In \cite{MermOrd}, the filter $\mathfrak{F}^n$ was defined smaller, hence a canonical colouring there is more restrictive. In this paper, we will not need that extra strength.
\end{numberedRemark}

The following theorem allows us, for our purposes, to assume that every arbitrary colouring we encounter is canonical.

\begin{theorem}[\hspace{-0.001pt}\cite{MermOrd}, Proposition 3.11]
For every natural $k$ and colouring $\col:\pairs{\omega^k}\to\set{0,1}$, there exists $X\subseteq \omega^k$, a subset of $\omega^k$ close in its supremum of order type $\omega^k$, such that the restriction of $\col$ to $X$ is a canonical colouring.
\end{theorem}

\begin{lemma}[\hspace{-0.001pt}\cite{MermOrd}, Lemma 4.3]
	\label{edgeColourScarce}
	Fix some canonical triangle-free $G=(\delta,E)$, where $\delta = \omega^k$ for some $k$ natural, with corresponding colouring $\col:\pairs{\delta}\to 2$. If there exists no independent $X\subseteq \delta$ closed in its supremum with $\otp{X} = \omega^2$, then the following statements hold
	\begin{enumerate}
		\item
		For a fixed $l$, there is at most one $j<l$ such that $\descolor(j,l) = 1$.
		\item
		For a fixed $j$, there is at most one $l$ such that $\domcolor(j,l) = 1$.
		\item
		For a fixed $l$, there is at most one $j$ such that $\domcolor(j,l) = 1$.
	\end{enumerate}
\end{lemma}

\begin{numberedRemark}
In the full definition of a canonical colouring of some arbitrary $\delta<\omega^\omega$, the functions $\descolor$, $\domcolor$ take additional parameters besides the $\CB$-rank of the ordinals --- indices of summands in the Cantor normal form of $\delta$. When colouring $\omega^k$, however, there is a unique summand in the Cantor normal form. With respect to the notation of the full definition, in this text $\descolor(j,l)$, $\domcolor(j,l)$ are shorthand for $\descolor(1,j,l)$, $\domcolor(1,j;1,l)$, respectively.
\end{numberedRemark}

\section{Upper bound}

	\begin{lemma}
		\label{omega^k neighbours}
	Let $G=(\delta, E)$ be a triangle free graph on some ordinal $\delta$. Let $A\subseteq \delta$ with $\otp{A} = \omega^k$ and let $B\subseteq\delta$ with $\otp{B}=\omega$ be such that $A\harass B$. Then there exists some $b\in B$ such that $\otp{\neighbours{b}\cap A} = \omega^k$.
	\end{lemma}
	
	\begin{proof}
	For each $a\in A$, let $m_a = \min(\neighbours{a}\cap B)$.
	
	Assume that for each $b\in B$ the set $\setarg{a\in A}{m_a > b}$ is cofinal in $A$. Then whenever $Y\subseteq A$ is finite, there exists some arbitrarily large $a\in A$ such that $\setarg{b\in B}{\max(\neighbours{Y}\cap B) < b <m_a}$ is not empty. Thus, extending $Y$ at each stage by such a sufficiently large element $a$, we can construct inductively a cofinal set $X\subseteq A$ such that $B\setminus\neighbours{X}$ is infinite in $B$. This contradicts $A\harass B$.
	
	Therefore, there must exist some $M\in B$ such that $\setarg{a\in A}{m_a > M}$ is not cofinal in $A$. In particular, the set $\setarg{a\in A}{m_a\leq M}$ contains a set of order type $\omega^k$. By additive indecomposability of $\omega^k$, there is some $b\leq M$ such that $\setarg{a\in A}{m_a = b}$ is of order type $\omega^k$. So $\otp{\neighbours{b}\cap A}=\omega^k$.
	\end{proof}

	\begin{fact}[Specker, \cite{Specker}]
	\label{SpeckerFact}
	$\omega^2\to (\omega^2,3)$, i.e.,
	every triangle-free graph on $\omega^2$ contains an independent set of order type $\omega^2$.
	\end{fact}

\begin{prop}
	\label{upperbound}
	$\closedramseynumber{\omega^2}{3}\leq \omega^6$
	\end{prop}
	\begin{proof}
	Let $G$ be a triangle-free canonical graph on $\omega^{6}$. Assume that
	\begin{itemize}
	\item[$(*)$] There is no independent copy of $\omega^2$ closed in its supremum in $G$.
	\end{itemize}
	By Lemma \ref{edgeColourScarce}
	\begin{itemize}
	\item 
	There is at most one $t<5$ such that $\descolor(5,t) = 1$;
	\item
	There is at most one $t<5$ such that $\domcolor(5,t) = 1$;
	\item
	There is at most one $t<5$ such that $\domcolor(t,5) = 1$.
	\end{itemize}
	
	Thus, there are $t_1<t_2<5$ such that $\descolor(5,t_j) = \domcolor(t_j,5)=\domcolor(5,t_j)=0$ for $j\in\set{1,2}$. For each $i\in\omega$, denote $h_i = \omega^5\cdot i$ and $W^{t_j}_i = \subtree{=t_j}{h_i}$.
	
	Since $\domcolor(5,t_j) = 0$, for each $i\in \omega$ and $j\in\set{1,2}$, there are at most finitely many $k>i$ such that $W^{t_j}_k\setminus \neighbours{h_i}$ is not large in $\subtree{=t_j}{h_k}$. Therefore, for every finite $I\subseteq \omega$, there are infinitely many $k\in\omega$ such that, for every $i\in I$, the set $W^{t_j}_k\setminus \neighbours{h_i}$ is large in $\subtree{=t_j}{h_k}$. Iteratively choosing in this manner a subset of $\setarg{h_i}{i\in\omega}$, we may assume that $W^{t_j}_k\setminus \neighbours{h_i}$ is large in $\subtree{=t_j}{h_k}$, for all $i<k\in\omega$ and $j\in\set{1,2}$.
	
	Now, as each $k\in\omega$ has finitely many $i<k$, and for each $i<k$ the set $W^{t_j}_k\setminus \neighbours{h_i}$ is large in $\subtree{=t_j}{h_k}$, we may thin out $\subtree{}{h_k}$ so that $W^{t_j}_k$ avoids $\neighbours{\setarg{h_i}{i<k}}$. Hence, we may assume there are no edges between $\set{h_i}$ and $W^{t_j}_k$, for all $i<k\in\omega$ and $j\in\set{1,2}$.
	
	\medskip
	\noindent\textbf{Claim.} Fix $j\in\set{1,2}$. Then $G$ can be thinned out so that for every $k\in\omega$, it holds that $W^{t_j}_k\oppress\setarg{h_i}{i>k}$.

\begin{proof}[proof of Claim]
	Assume the contrary.
	
	We will construct inductively $k_n,X_n,I_n$ such that at every stage:
	\begin{enumerate}
		\item
		$I_n\subseteq\omega$, $k_n\in I_n$, $X_n\subcof W^{t_j}_{k_n}$;
		\item
		$I_{n+1}\subseteq\setarg{i\in I_n}{i>k_n}$;
		\item
		$\neighboursInSet{X_n}{\setarg{h_i}{i\in I_{n+1}}} = \emptyset$.
	\end{enumerate}
	Set $I_0=\omega$. We describe the inductive step, given some $I_n\subseteq \omega$:
	
	Thin out $G$ so that $\subfan{\omega^6} = \setarg{h_i}{i\in I_n}$. By assumption, there is some $k\in I_n$ such that $W^{t_j}_k\not\oppress \setarg{h_i}{i>k}$. Fix $k_n$ to be such a $k$ and let $X_n\subcof W^{t_j}_{k_n}$ be of order type $\omega$ such that $\setarg{h_i}{i\in I_n}\setminus\neighbours{X_n}$ is infinite. Define $I_{n+1}=\setarg{i\in I_n}{i>k_n,~ h_i\notin\neighbours{X_n}}$.
	
	Now, let $Y = \setarg{h_{k_n}}{n\in\omega}$ and let $X = \bigcup_{n\in\omega} X_n$. Observe that $X\cup Y$ is a copy of $\omega^2$ closed in its supremum, with no edges crossing between $X$ and $Y$. By Ramsey's theorem and triangle-freeness, $Y$ has an infinite independent subset. By thinning out $X\cup Y$, we may assume that already $Y$ is independent. By Fact \ref{SpeckerFact}, $X$ has an independent subset of order type $\omega^2$. Again by thinning out $X\cup Y$ we may assume $X$ is independent. Hence, by construction, $X\cup Y$ is independent in contradiction to $(*)$.
\end{proof}
Assume that $G$ was thinned out as guaranteed by the claim. Observe that since $\domcolor(t_j,5)=0$, in fact $W^{t_j}_k\harass\setarg{h_i}{i>k}$ for any $j\in\set{1,2}$ and $k\in\omega$.

Fix some $k$. Choose arbitrarily some $A_2\subseteq W^{t_2}_k$ and $A_1\subseteq W^{t_1}_k$ such that both $A_1$ and $A_2$ are cofinal in $\subtree{}{h_k}$, and $A_1\cup A_2$ is a copy of $\omega^2$ closed in its supremum.

Since $A_2\harass\setarg{h_i}{i>k}$, by Lemma \ref{harassmentLemma} we may thin out so that $\neighbours{h_i}\cap A_2$ is cofinite in $A_2$ for every $i>k$. By Lemma \ref{omega^k neighbours}, there exists some $i>k$ such that $\neighbours{h_i}\cap A_1$ is of order type $\omega^2$. Then $\neighbours{h_i}\cap (A_1\cup A_2)$ contains an independent copy of $\omega^2$ closed in its supremum, which concludes the proof.
\end{proof}

\section{Lower bound}

For any ordinal $\alpha<\omega^\omega$, for each $i$, denote by $\alpha_i\in\omega$ the coefficient of $\omega^i$ in the Cantor normal form of $\alpha$. That is, $\alpha=\sum_{i\in\omega^*}\omega^i\cdot \alpha_i$. Denote $\layer{}{n} = \subtree{=n}{\omega^\omega}$. For every natural $n$, consider the sets of edges:
\begin{gather*}
E^n_1 = \setarg{(\alpha,\beta)\in \layer{}{n-1}\times \layer{}{n}}{\alpha\treedescendant\beta}
\\
E^n_2 = \setarg{(\alpha,\beta)\in \layer{}{n}\times \layer{}{n-2}}{\alpha<\beta}
\\
E^n_3 = \setarg{(\alpha,\beta)\in \layer{}{n-3}\times \layer{}{n}}{\alpha<\beta,\alpha\not\treeorder\beta, \max\set{\beta_i} < \alpha_{n-1}}
\\
E^n_4 = \setarg{(\alpha,\beta)\in\layer{}{n-4}\times \layer{}{n}}{\alpha<\beta, \alpha\not\treeorder\beta, \max\set{\beta_i} >\alpha_{n-1} + \alpha_{n-2}}
\end{gather*}
Let $G_\omega = (V_\omega,E_\omega)$ be the graph on $V_\omega = \omega^\omega$ with edges $E_\omega = \bigcup_{\substack{n\in\omega \\ i\leq 4}} E^n_i\cup (E^n_i)^{-1}$.

\begin{lemma}
The graph $G_{\omega}$ is triangle-free.
\end{lemma}

\begin{proof}
Assume to the contrary that $\set{\alpha,\beta,\gamma}$ is a triangle. Without loss of generality assume $\CBr{\alpha} < \CBr{\beta} < \CBr{\gamma} = n$. Clearly $\CBr{\alpha}\geq n-4$.

Consider the case $\CBr{\beta} = n-1$, meaning $\beta\treedescendant\gamma$. We cannot have $\alpha\treeorder\gamma$, hence $\alpha\not\treeorder \beta$ and in particular $\CBr{\alpha}\neq\CBr{\beta} -1$. So $\CBr{\alpha} < n-2$ and it must be that $\alpha<\gamma$. If $\beta < \alpha$, we will have $\beta<\alpha<\gamma$ and $\beta\treeorder\gamma$, implying $\alpha\treeorder \gamma$, which is false. So $\alpha<\beta$, resulting in $\CBr{\alpha}\neq \CBr{\beta} -2$. So $\CBr{\alpha}=n-4$. Thus, $\max\set{\beta_i} < \alpha_{n-2} \leq \alpha_{n-1} + \alpha_{n-2} < \max\set{\gamma_i}$. But $\beta\treeorder\gamma$, therefore $\max\set{\beta_i} \geq \max\set{\gamma_i}-1$ in contradiction.

Consider the case $\CBr{\beta} = n-2$. Then $\alpha < \gamma$ and $\gamma <\beta$. If $\alpha\treeorder\beta$, then we would have $\gamma<\alpha$, so $\alpha\not\treeorder\beta$ and $\CBr{\alpha} \neq n-3$. Since $\beta \nless \alpha$, we have $\CBr{\alpha}\neq n-4$. So $\CBr{\alpha}$ cannot take any value, a contradiction.

We are left with $\CBr{\beta} = n-3$ and $\alpha\treedescendant\beta$. But now  $\alpha_{n-1} + \alpha_{n-2} < \max\set{\gamma_i} < \beta_{n-1}$, despite $\alpha_{n-1} = \beta_{n-1}$. We conclude that there are no triangles in $G_{\omega}$.
\end{proof}

\begin{notation}
For $X\subseteq \omega^\omega$, denote $\CBr{X} = \sup\set{\CBr{\alpha}}_{\alpha\in X}$.
\end{notation}

\begin{lemma}
If $X\subseteq G_{\omega}$ is an independent copy of $\omega^{k+1}$, closed in its supremum, with $X$ not cofinal in $\omega^\omega$, then $\CBr{X} \geq 5k$.
\end{lemma}

\begin{proof}
We prove by induction on $k$. Let $X\subseteq G_\omega$ be an independent, closed in its supremum, copy of $\omega^{k+1}$ with $\sup X < \omega^\omega$. Let $\ordPersFunc{X}:\omega^{k+1} \to X$ be the bijection witnessing $\omega^{k+1}\cong X$. Due to $X$ being closed in its supremum, $\ordPersFunc{X}$ is continuous. Since $X$ is bounded, $\CBr{X}$ is finite.

For each $i<k+1$, we may consider $\CB\circ\ordPersFunc{X}$ as a colouring of $\subtree{=i}{\omega^{k+1}}$ in $\CBr{X}$ many colours. By iterating the pigeonhole principle, in $\subtree{=i}{\omega^{k+1}}$ there is a copy of $\omega^{k+1-i}$ on which $\CB\circ\ordPersFunc{X}$ is constant. By indecomposability of $\omega^{k+1-i}$, we may thin out $\omega^{k+1}$ (and $X$, accordingly) so that $\CB\circ\ordPersFunc{X}$ is constant on $\subtree{=i}{\omega^{k+1}}$. Iterating this thinning out $k+1$ times, we may assume $\CB\circ\ordPersFunc{X}$ is constant on $\subtree{=i}{\omega^{k+1}}$ for each $i<k+1$.
	
Denote $\alpha^j=\ordPersFunc{X}(\omega^k\cdot j)$, $n=\CBr{\alpha^1}$, $B_j = \ordPersFunc{X}[\subfan{\omega^k\cdot j}]$, and $m=\CBr{\beta}$ for some $\beta\in B_1$. Observe that by continuity of $\ordPersFunc{X}$, whenever $A\subseteq\omega^{k+1}$ with $\sup A=a\in \omega^{k+1}$, the set $\ordPersFunc{X}[A]$ must intersect $\subtree{}{\ordPersFunc{X}(a)}$.

\medskip
\noindent\textbf{Claim.} $m\leq n-5$.
\begin{proof}[proof of claim]
By $\subfan{\alpha^j}\subseteq\neighbours{\alpha^j}$, it cannot be that $B_j$ intersects $\subfan{\alpha^j}$. Therefore, $m = \CBr{B_j}\neq n-1$.

Similarly, since $\subtree{=n-2}{\alpha^{j+1}}\subseteq \neighbours{\alpha^j}$, we have $m = \CBr{B_{j+1}}\neq n-2$.

For any $C\subcof\subtree{=n-3}{\alpha^j}$, the set $\setarg{\gamma_{n-1}}{\gamma\in C}$ is unbounded in $\omega$. hence, there exists some $\gamma\in C$ with $\alpha^{j+1}\in\neighbours{\gamma}$. Thus, $B_j\cap \subtree{=n-3}{\alpha^j}$ cannot be a cofinal subset of $\subtree{=n-3}{\alpha^j}$ and so $m \neq n-3$.

Let $\gamma\in\subtree{=n-4}{\alpha^j}$. Since $X$ is not cofinal in $\omega^\omega$, there is some $r$ large enough such that $X\subseteq \subtree{}{\omega^r}$. There are only finitely many elements $\alpha\in \subtree{=n}{\omega^r}$ with $\max\set{\alpha_i} \leq \gamma_{n-1} + \gamma_{n-2}$, so $\neighboursInSet{\gamma}{\setarg{\alpha^j}{j\in\omega}} \neq \emptyset$. We conclude that $B_j$ is disjoint from $\subtree{=n-4}{\alpha^j} = \emptyset$ and so $m\neq n-4$.
\end{proof}

For each $j\in\omega$, the set $X\cap \subtree{}{\alpha^j}\setminus\set{\alpha^j}$ contains an independent, closed in its supremum, copy of $\omega^k$, which is not cofinal in $\omega^\omega$. So by the induction hypothesis, $m\geq 5(k-1)$. Combined with the claim, this results in $n\geq 5k$.
\end{proof}

\begin{corollary}
	\label{lowerbound}
$R^{cl}(\omega^{k+1},3) \geq \omega^{5k+1}$
\end{corollary}

\begin{proof}
Consider the subgraph induced by $G_\omega$ on $\omega^{5k+1}=\subtree{}{\omega^{5k+1}}\setminus \set{\omega^{5k+1}}$, call this subgraph $G$. Let $X$ be some independent, closed in its supremum copy of $\omega^{k+1}$ in $G$. Every final segment $X'\subseteq X$ of $X$ contains an independent, closed in its supremum copy of $\omega^{k+1}$, so by the above lemma $\CBr{X'} \geq 5k$. Then $X\subcof \omega^{5k+1}$ and in particular $X$ is not contained in the restriction of $G$ to any $\delta<\omega^{5k+1}$. Thus, for every $\delta<\omega^{5k+1}$, we have found a graph demonstrating $R^{cl}(\omega^{k+1},3) > \delta$.
\end{proof}

\begin{theorem}
	$\closedramseynumber{\omega^2}{3} = \omega^6$
\end{theorem}

\begin{proof}
	Apply Corollary \ref{lowerbound} to $k=1$ and combine with Proposition \ref{upperbound}.
\end{proof}

\section{Acknowledgments}
The research was conducted at Ben-Gurion University of the Negev. The research was partially supported by ISF grant No. 181/16 and 1365/14.

\bibliographystyle{alpha}
\bibliography{myrefs}
\end{document}